\documentclass[twoside,11pt,reqno]{amsart}
\usepackage[dvipsnames]{xcolor}
\usepackage{amsthm,amsmath,amsfonts,graphicx,geometry,lipsum,amsxtra}
\usepackage{hyperref}
\usepackage{mathrsfs}
\usepackage{verbatim}
\usepackage{tikz}
\usepackage{epigraph}

\numberwithin{equation}{section}
\usepackage{esint}
\usepackage{mathtools}

\usepackage[dvipsnames]{xcolor}
\usepackage{etoolbox}
\hypersetup{
    colorlinks=true,
    linkcolor=blue,
    citecolor=blue,
    urlcolor=red,
    pdftitle={[BKP23] Determination of lower order perturbations of a polyharmonic operator in two dimensions},
    }
\usepackage{comment}
\usepackage{amsmath}
\usepackage{tikz}
\usepackage{epigraph}
\usepackage{lipsum}
\allowdisplaybreaks

\definecolor{titlepagecolor}{cmyk}{1,.60,0,.40}
\patchcmd{\subsection}{\normalfont}{\normalfont\color{blue}}{}{}
\DeclareFixedFont{\titlefont}{T1}{ppl}{b}{it}{0.5in}

\def\th@plain{%
  \thm@notefont{}
  \itshape 
}
\def\th@definition{%
  \thm@notefont{}
  \normalfont 
}
\makeatother

\DeclareMathOperator\bd{\partial\Omega}

\def\tre{\textcolor{black}}

\usepackage{srcltx} 
\usepackage{amscd}
\usepackage{amssymb}
\usepackage{amsmath}
\usepackage{pb-diagram}
\usepackage{graphicx}
\usepackage{hyperref}
\usepackage{array}
\usepackage[all]{xy}
\xyoption{matrix}
\xyoption{arrow}
\usepackage{pdfsync}
\usepackage{physics}
\usepackage{enumitem}
 
\usepackage{thmtools}
\usepackage{amssymb}

\usepackage[none]{hyphenat}

\theoremstyle{definition}
\newtheorem{defn}{Definiton}[section]
\newtheorem{theorem}[defn]{Theorem}
\newtheorem{cor}[defn]{Corollary}
\newtheorem{lemma}[defn]{Lemma}
\newtheorem{prop}[defn]{Proposition}

\newtheorem{remark}[defn]{Remark}
\calclayout

\begin{document}
\title[Inverse problem for polyharmonic operator in two dimensions]{Determination of lower order perturbations of a polyharmonic operator in two dimensions}
\subjclass[2020]{35R30; 35J40  }
\keywords{Inverse boundary value problem; Perturbed two-dimensional polyharmonic operator; Cauchy data; Uniqueness. }

\author[Bansal]{Rajat Bansal}
\address{Rajat Bansal\\ Department of Mathematics, Indian Institute of Science Education and Research, Pune}
\email{rajat.bansal@students.iiserpune.ac.in}

\author[Krishnan]{Venkateswaran P. Krishnan}
\address{Venkateswaran P. Krishnan\\ Centre for Applicable Mathematics, Tata Institute of Fundamental Research, Bengaluru.}
\email{vkrishnan@tifrbng.res.in}

\author[Pattar]{Rahul Raju Pattar}
\address{Rahul Raju Pattar\\ Centre for Applicable Mathematics, Tata Institute of Fundamental Research, Bengaluru.}
\email{rahul22@tifrbng.res.in}

\vspace*{-1cm}

 \maketitle
    
\begin{abstract}
We study an inverse boundary value problem for a polyharmonic operator in two dimensions. We show that the Cauchy data uniquely determine all the anisotropic perturbations of orders at most $m-1$ and several perturbations of orders $m$ to $2m-2$ under some restriction. The uniqueness proof relies on the $\bar\partial$-techniques and the method of stationary phase.
\end{abstract}

\section{Introduction}

\noindent Let $\Omega$ be a simply connected bounded domain with smooth boundary in $\mathbb{R}^2$. This paper aims to study the inverse boundary value problem in two dimensions for a polyharmonic operator of the form: 
\begin{equation}\label{sec1:Eq1}
  \mathcal{L} =  \partial^{m}\bar{\partial}^{m} + \sum_{j,k = 0}^{m-1}{A_{j,k}\partial^{j}\bar{\partial}^{k}}, \quad m\geq 2.
\end{equation} 
Note that $x = (x_1,x_2) \in \Omega \subset \mathbb{R}^2$ is identified with $z= x_1+ i x_2 \in \mathbb{C}$ and  $\partial = \frac{1}{2} \left( \partial_{x_1} - i \partial_{x_2} \right),$ $\bar \partial = \frac{1}{2} \left( \partial_{x_1} + i \partial_{x_2} \right)$. \\

 \noindent The polyharmonic operators have applications in physics and geometry, for example the Kirchhoff-Love plate equation in the theory of elasticity and the study of the Paneitz-Branson operator in conformal geometry; see \cite{GFC91}.\\

 \noindent  In this paper we uniquely determine the coefficients $A_{j,k}$ from the set of Cauchy data
  \begin{equation}\label{CL2}
		\mathcal{C}(\mathcal{L}) = \left\{  \left( u|_{\partial\Omega}, \partial_{\nu}u|_{\partial\Omega}, \partial_{\nu}^{2}u|_{\partial\Omega}  \ldots, \partial_\nu^{(2m-1)}u|_{\partial\Omega}   \right):  u \in H^{2m}(\Omega),  \mathcal{L}u =0\right\},
\end{equation}
\noindent where $\nu$ is an outer unit normal to $\bd$.\\

\noindent To the best of our knowledge, there are few prior works investigating inverse problems for lower order perturbation of polyharmonic operator (especially for $m\geq 3$) in two dimensions compared to higher dimensions. For $m=2$, in \cite{Ik91} Ikehata proved a local uniqueness of a potential $V\in L^2(\Omega)$ associated to the Calder\'{o}n problem for $(\Delta^2+V)u=0$ with $\|V\|_{L^2}$ small. In \cite{Ik93}, the same author studied a local uniqueness theorem for the Calder\'{o}n problem for a perturbed  biharmonic operator related to Love-Kirchhoff plate theory. In \cite{Ik94}, the author studied the relationship between two D-N maps in the linear theory of elasticity in two dimensions and the linearized Calderon problem for anisotropic fourth order equation. In \cite{IY15, LUW15}, the authors have studied Navier-Stokes equation in two dimensions using the biharmonic operator.\\ 

 \noindent In dimensions $n \geq 3,$ Krupchyk, Lassas, and Uhlmann \cite{KLU14} established that the Cauchy data for a polyharmonic operator uniquely determines first order perturbations. This work was extended by many authors; see, for instance, \cite{Gk, BG19,BKS21, BG}. Till now, the perturbations considered for the polyharmonic operator are of order at most $m$ in $n \geq 3$. Moreover, in \cite{SS23}, the authors considered the linearized Calder\'on inverse problem for polyharmonic operator and recovered several lower order coefficients up to order $2m-1$ in $n\geq 3$.\\
 
\noindent In this paper, we establish that the Cauchy data for a polyharmonic operator in two dimensions uniquely determines all anisotropic perturbations of order at most $m-1$ and  several perturbations of orders $m$ to $2m-2$ with some restrictions. This restriction is captured in the following representation of the operator $\mathcal{L}$ as
\[(\partial \bar{\partial})^{m} + A_{m-1,m-1}(\partial\bar{\partial})^{m-1} + \sum_{l=1}^{m-2}{\left(\sum_{j+k=m-l-1}{A_{j+l,k+l}\partial^j\bar\partial^k}\right)(\partial\bar\partial)^l } + \sum_{l=0}^{m-1}\sum_{j + k = l}{A_{j,k}\partial^{j}\bar{\partial}^{k}}.
\]
The constraint on the coefficients of orders $m$ to $2m-2$ is required for the techniques employed in this paper to work, mainly, to make the equation for the amplitude of complex geometric optics (CGO) solutions to be independent of the coefficients.\\

 \noindent An early study of inverse boundary value problems for second-order operators was carried out by Calder\'{o}n \cite{C}.  Sylvester and Uhlmann \cite{SU} obtained the first uniqueness result for the Calder\'{o}n problem. In two dimensions, the Calder\'on problem was studied by Nachman \cite{N}, Brown-Uhlmann \cite{BU}, and finally Astala–P\"{a}iv\"{a}rinta \cite{AP}. Nachman required two derivatives to convert the conductivity equation into the Schrödringer equation. The paper of Astala and P\"{a}iv\"{a}rinta solved Calderón’s problem most generally for $L^\infty$ conductivity. \\
 
 \noindent Our approach relies on two main techniques - the $\bar\partial$-techniques and the method of stationary phase. These techniques were first used by Bukhgeim in his seminal work \cite{Buk} to recover the zeroth order perturbation of the Laplacian in two dimensions that has led to many developments in the study of two-dimensional inverse boundary value problems. However, his proof only gives uniqueness for potentials in the class $W^{1,p}$, $p > 2$ as pointed out in Bl\r{a}sten’s licentiate thesis \cite{EB}. The corresponding partial data problem was considered by Imanuvilov, Uhlmann and Yamamoto in  \cite{IUY} and Guillarmou and Tzou in \cite{GT2}. The inverse problem for magnetic Schr\"{o}dinger equation in $n = 2$ using Bukhgeim approach was studied by Lai in \cite{lai2011} and for partial data by Imanuvilov, Uhlmann and Yamamoto in \cite{IUY2012} and by Tzou in  \cite{tzou}. We refer the interested reader to the survey article \cite{GT13} for a detailed overview of results and techniques for inverse problems in two dimensions.\\

\noindent Now we state the main theorem of this paper.

\begin{theorem}\label{MT}
Let $\Omega$ be a simply connected bounded domain with smooth boundary in $\mathbb{R}^2$.
Let $\mathcal{L}$ and $\tilde{\mathcal{L}}$ be two operators  of the form \eqref{sec1:Eq1} with coefficients $A_{j,k} , \tilde A_{j,k} \in W^{j+ k + 1,p}(\Omega),\ p>2$, respectively. Assume that 
\begin{equation}\label{zerocond}
  \partial_{\nu}^{l}{A}_{j,k} = \partial_{\nu}^{l}{\tilde{A}}_{j,k} \text{ and } {A}_{0,k} = \tilde{A}_{0,k}  \text{ on } \partial\Omega, \quad \text{ for } 0 \leq l \leq j -1, \ 0\leq j,k \leq m-1 .
\end{equation}
Then $\mathcal{C}(\mathcal{L}) = \mathcal{C}(\tilde{\mathcal{L}})$ implies that $A_{j,k} = \tilde{A}_{j,k}$ on $\Omega$ for $0 \leq j,k \leq m-1.$
\end{theorem}

\noindent We require the condition \eqref{zerocond} on the coefficients to make the boundary terms \emph{zero} when we apply integration by parts on the integral identity and also to  apply the method of stationary phase as explained in Section \ref{III}.\\

\noindent  Alternatively, one can consider a Cauchy data set of the form
\begin{equation*}
		\mathcal{N}(\mathcal{L}) = \left\{  \left( u|_{\partial\Omega}, \ldots,   (-\Delta)^{m-1} u|_{\partial\Omega}, (\partial_\nu u)|_{\partial\Omega}, \cdots \partial_\nu (-\Delta)^{m-1} u|_{\partial\Omega}   \right):  u \in H^{2m}(\Omega),  \mathcal{L}u =0\right\},
\end{equation*}
as $\mathcal{C}(\mathcal{L})$ can be obtained from $\mathcal{N}(\mathcal{L})$ by an explicit description for the Laplacian in the boundary normal coordinates, see \cite{LU}. We have the following corollary.

\begin{cor}\label{MT2}
With the hypothesis as in Theorem \ref{MT}, $\mathcal{N}(\mathcal{L}) = \mathcal{N}(\tilde{\mathcal{L}})$ implies that $A_{j,k} = \tilde{A}_{j,k}$ on $\Omega$ for $0 \leq j,k \leq m-1.$
\end{cor}
To prove Theorem \ref{MT}, we need to construct so-called complex geometric optics (CGO) solutions. The next theorem gives the existence of such solutions in our setting. Let us fix some notation before stating the theorem. Let $\Phi = i(z-z_0)^2$ where $z_0 \in \Omega$ and $dS$ be the surface measure on $\bd$. 
 
\begin{theorem}\label{CGO}
    Let $a$ be smooth function such that $\bar\partial^m a =0$ in $\Omega$. If $A_{j,k} \in W^{j+k+1,p}(\Omega)$, for some $p>2$, then for all small $h>0$ there exist solutions $u \in H^{2m}(\Omega)$ to 
\begin{equation}\label{scl}
\mathcal{L} u = \partial^{m}\bar{\partial}^{m}u + \sum_{j,k = 0}^{m-1}{A_{j,k}\partial^{j}\bar{\partial}^{k}u} = 0, \quad \text{ in } \Omega,
\end{equation}
    of the form 
    \begin{equation}\label{soln}
		u = e^{\Phi/h}(a + r_h), 
	\end{equation}   
    where the correction term $r_h$ satisfies $\|r_{h}\|_{H^{m}(\Omega)} =  O(h^{\frac{1}{2} + \epsilon})$ for some $\epsilon > 0$.
\end{theorem}


\section{CGO Solutions} 
\noindent In this section, we prove Theorem \ref{CGO} as stated in Introduction. {\color{black} For this we need Cauchy operators - $\partial^{-1}, \bar{\partial}^{-1}$ which are defined by 
$$
\partial^{-1} f(z)=\frac{1}{\pi} \int_{\Omega} \frac{f(\xi)}{\bar{z}-\bar{\xi}} d\xi \quad  \text{ and }
\quad \bar{\partial}^{-1} f(z)=\frac{1}{\pi} \int_{\Omega} \frac{f(\xi)}{z-\xi} d\xi,
$$
for all $z$ where the integrals converge. The mapping properties of these operators are listed in the following lemma. For the proofs of these properties, we refer to \cite[Chapter I]{Vekua}.

\begin{lemma}\label{MPCO}
    The operators $\partial^{-1}$ and $\bar{\partial}^{-1}$ have the following properties:
\begin{enumerate}
    \item If $f \in L^1(\Omega)$ then $\partial \partial^{-1} f=\bar{\partial} \bar{\partial}^{-1} f=f$.

    \item The operators $\partial^{-1}, \bar{\partial}^{-1}$ are bounded in $L^p(\Omega), 1<p<\infty$. 
      
    \item If $1<p<\infty$, then $\partial^{-1}, \bar{\partial}^{-1}: L^p(\Omega) \rightarrow W^{1,p}(\Omega)$ are bounded.
      
\end{enumerate}
\end{lemma} 
}

\noindent Now we begin the construction of CGO solutions by writing 
 \[
    \mathcal{L}u  = \partial^{m}\bar{\partial}^{m}u + \sum_{j,k = 0}^{m - 1}{A_{j,k}\partial^{j}\bar{\partial}^{k}u} = 0
 \]   
 in the following form
\begin{equation}\label{neq}
\mathcal{L}u = \partial^{m}\bar{\partial}^{m}u + \sum_{j,k = 0}^{m - 1}{\partial^{j}(A_{j,k}'\bar{\partial}^{k}u)} = 0,
\end{equation}
\noindent where we can define $A'_{j,k} \in W^{j+k+1,p}(\Omega) $ uniquely satisfying
\begin{equation}\label{eqn7}
A_{j,k} = \sum_{l = j}^{m-1}{\binom{l}{j}\partial^{l-j}A'_{l,k}}.
\end{equation}

\noindent Substituting $u=e^{\Phi/h}f$ in \eqref{neq}, we have
\[\partial^{m}\left(e^{\Phi/h}\bar{\partial}^{m}f\right)  + \sum_{j,k = 0}^{m-1}{\partial^{j}\left(e^{\Phi/h}A'_{j,k}\bar{\partial}^{k}f\right)} = 0.\]
Now, we write $G = e^{\Phi/h}\bar\partial^mf$ and the above expression takes the form
\begin{align}
\begin{split}
 \bar{\partial}^{m}f & = e^{-\Phi /h}G\\
 \partial^{m}G & = -\sum_{j,k = 0}^{m-1}{\partial^{j}\left(e^{\Phi/h}A'_{j,k}\bar{\partial}^{k}f\right)}
 \end{split}
\end{align}
The problem here is that $|e^{\pm\Phi/h}|$ grows too fast when $h \to 0$. This can be solved by choosing $G= e^{\bar\Phi/h}g$ to get
\begin{align}
    \bar{\partial}^{m}f & = e^{(\bar{\Phi}-\Phi)/h}g \label{est1} \\
    \partial^{m}g & = -\sum_{j,k = 0}^{m-1}{\partial^{j}\left(e^{(\Phi - \bar{\Phi})/h}A'_{j,k}\bar{\partial}^{k}f\right)} \label{est2}
\end{align}

\noindent For \eqref{est1}, we take the solution 
\begin{equation}\label{fsoln}
f  = a + \bar{\partial}^{-m}\left(e^{(\bar{\Phi}-\Phi)/h}g\right), \quad \text{ where } \bar\partial^m a = 0 
\end{equation}
 and for \eqref{est2} we choose the solution
\begin{equation}\label{gsoln}
g  = -\sum_{j,k = 0}^{m-1}{\partial^{j-m}\left(e^{(\Phi - \bar{\Phi}) /h} A'_{j,k}\bar{\partial}^{k}f \right)}.
\end{equation}
By combining these two, we get an integral equation for $g$ of the form
\begin{equation}\label{eqng}
    g + \sum_{j,k = 0}^{m-1}{\partial^{j-m}\left(e^{(\Phi - \bar{\Phi}) /h}A'_{j,k}\bar{\partial}^{k-m}\left(e^{(\bar{\Phi}-\Phi)/h}g\right)\right)}  = -\sum_{j,k = 0}^{m-1}{\partial^{j-m}\left(e^{(\Phi - \bar{\Phi}) /h}A'_{j,k}\bar{\partial}^{k}a\right)}.    
\end{equation}
The above expression for $g$ can be written in the form
\begin{equation}\label{CT3}
    (I - \mathcal{S}_h)g = w,
\end{equation}
where 
\begin{align}\label{estimate5}
\begin{split}
    \mathcal{S}_h(v) &=  - \sum_{j,k = 0}^{m-1}{\partial^{j-m}\left(e^{(\Phi - \bar{\Phi}) /h}A'_{j,k}\bar{\partial}^{k-m}\left(e^{(\bar{\Phi}-\Phi)/h}v\right)\right)} \\
    w &= -\sum_{j,k = 0}^{m-1}{\partial^{j-m}\left(e^{(\Phi - \bar{\Phi}) /h}A'_{j,k}\bar{\partial}^{k}a\right)}.
\end{split}
\end{align}
 The existence of a CGO solution of the form \eqref{soln} to the equation \eqref{scl} depends on the solvability of \eqref{CT3}. To this end, we estimate the norm of $\mathcal{S}_{h}$ for which we need the following crucial operator bound from \cite[Lemma 2.3]{GT} and \cite[Lemma 5.4]{GT13}.
\begin{lemma}\label{inverseestimate}
Let $q\in (1,\infty)$ and $p > 2$, then there exists $C > 0$ independent of $h$ such that for all $\omega \in W^{1,p}(\Omega)$
\begin{align*}
\|{\partial}^{-1}(e^{(\Phi - \bar{\Phi})/h}\omega)\|_{L^{q}(\Omega)} & \leq C h^{2/3}\|\omega\|_{W^{1,p}(\Omega)} \quad \text{if } 1 < q < 2,\\
\|{\partial}^{-1}(e^{(\Phi - \bar{\Phi})/h}\omega)\|_{L^{q}(\Omega)} & \leq C h^{1/q}\|\omega\|_{W^{1,p}(\Omega)} \quad \text{if } 2 \leq q \leq p.
\end{align*}
There exist $\epsilon > 0$ and $C > 0$ such that for all $\omega \in W^{1,p}(\Omega)$
\[\|{\partial}^{-1}(e^{(\Phi - \bar{\Phi})/h}\omega)\|_{L^{2}(\Omega)}  \leq C h^{\frac{1}{2} + \epsilon}\|\omega\|_{W^{1,p}(\Omega)}.\]
\end{lemma}
\begin{lemma}\label{estimate11}
For any $1 < r  \leq p$, the operator $\mathcal{S}_{h}$ is bounded on $L^{r}(\Omega)$ and satisfies $\norm{\mathcal{S}_{h}}_{L^{r}\to L^{r}} = O(h^{1/r})$ for $r > 2$ and $\norm{\mathcal{S}_{h}}_{L^{2}\to L^{2}} = O(h^{\frac{1}{2} - \epsilon})$ for any $0 < \epsilon < 1/2$ small.     
\end{lemma}
\begin{proof}
\noindent Firstly, for $2< r \leq p,$ we obtain
\begin{align*}
 \|S_{h}(v)\|_{L^{r}(\Omega)} & \leq \sum_{j,k = 0}^{m-1} \norm{\partial^{j-m}\left(e^{(\Phi - \bar{\Phi})/h}A'_{j,k}\bar{\partial}^{k-m}\left(e^{(\bar{\Phi} - \Phi)/h}v\right)\right)}_{L^{r}(\Omega)}. \\
 & \leq C\sum_{j,k = 0}^{m-1}\norm{\partial^{-1}\left(e^{(\Phi - \bar{\Phi})/h}A'_{j,k}\bar{\partial}^{k-m}\left(e^{(\bar{\Phi}-\Phi)/h}v\right)\right)}_{L^{r}(\Omega)} 
 \text{\tre{(using Lemma \ref{MPCO}-(2))}}\\
 & \leq C h^{\frac{1}{r}}\sum_{j,k = 0}^{m-1}\norm{A'_{j,k}\bar{\partial}^{k-m}\left(e^{(\bar{\Phi} - \Phi)/h}v\right)}_{W^{1,r}(\Omega)} 
 \hspace{1.8cm}\text{\tre{ (using Lemma \ref{inverseestimate})}} \\
& \leq C h^{\frac{1}{r}}\sum_{k = 0}^{m-1}\norm{\bar{\partial}^{k-m}\left(e^{(\bar{\Phi} - \Phi)/h}v\right)}_{W^{1,r}(\Omega)} \hspace{2.6cm}\text{\tre{(since $A'_{j,k} \in L^{\infty}(\Omega)$)}}\\
& \leq C h^{\frac{1}{r}}\|v\|_{L^{r}(\Omega)}.
\end{align*}

\noindent Further, for $1<r<2,$ 
\begin{align*}
 \|S_{h}(v)\|_{L^{r}(\Omega)} & \leq \sum_{j,k = 0}^{m-1} \norm{\partial^{j-m}\left(e^{(\Phi - \bar{\Phi})/h}A'_{j,k}\bar{\partial}^{k-m}\left(e^{(\bar{\Phi} - \Phi)/h}v\right)\right)}_{L^{r}(\Omega)} \\
 & \leq C\sum_{j,k = 0}^{m-1}\norm{\partial^{-1}\left(e^{(\Phi - \bar{\Phi})/h}A'_{j,k}\bar{\partial}^{k-m}\left(e^{(\bar{\Phi}-\Phi)/h}v\right)\right)}_{L^{r}(\Omega)} \text{\tre{ (using Lemma \ref{MPCO}-(2))}}\\
 & \leq C \sum_{j,k = 0}^{m-1}\norm{A'_{j,k}\bar{\partial}^{k-m}\left(e^{(\bar{\Phi} - \Phi)/h}v\right)}_{L^{r}(\Omega)}
 \hspace{2.6cm} \text{\tre{ (using Lemma \ref{MPCO}-(2))}}\\
& \leq C \sum_{k = 0}^{m-1}\norm{\bar{\partial}^{k-m}\left(e^{(\bar{\Phi} - \Phi)/h}v\right)}_{L^{r}(\Omega)}  \hspace{3.4cm}\text{\tre{(since $A'_{j,k} \in L^{\infty}(\Omega)$)}}\\
& \leq C \|v\|_{L^{r}(\Omega)}. \hspace{6.9cm} \text{\tre{ (using Lemma \ref{MPCO}-(2))}}
\end{align*}
For all $\varepsilon>0$ small, interpolating between $r = 1+ \varepsilon$ and $r = 2+ \varepsilon$, gives the desired result for $r = 2$.
\end{proof}
\begin{prop}\label{g}
For all sufficiently small $h > 0$, there exist a solution $g\in H^{m}(\Omega)$ to the equation 
\[
    (I - \mathcal{S}_{h})g = w,
\]
where $\mathcal{S}_{h}$ and $w$ defined in \eqref{estimate5} which satisfies $\|g\|_{L^{2}} = O(h^{\frac{1}{2} + \epsilon})$.  
\end{prop}
\begin{proof}
    In view of Lemma \ref{estimate11}, equation \eqref{CT3} can be solved by using Neumann series by setting (for small $h > 0$)
\[
    g = \sum_{j=0}^{\infty} \mathcal{S}_h^j w.
\]
as an element of $L^{2}(\Omega)$. Indeed $\|w\|_{L^{2}(\Omega)} = O(h^{\frac{1}{2} + \epsilon})$ by Lemma \ref{inverseestimate} and $\norm{\mathrm{S}_{h}}_{L^{2}\rightarrow L^{2}} = O(h^{\frac{1}{2} - \epsilon})$ by Lemma \ref{estimate11} we obtain $\norm{\mathrm{S}_{h}^{j}w}_{L^{2}} = O\left(h^{(\frac{1}{2} - \epsilon)j}h^{\frac{1}{2} + \epsilon}\right)$ which implies $\norm{g}_{L^{2}(\Omega)} = O(h^{\frac{1}{2} + \epsilon})$. \\

\noindent \tre{To show that $g$ is in $H^m(\Omega)$, we rewrite \eqref{eqng} as
\begin{equation}\label{eqng2}
    g = - \sum_{j,k = 0}^{m-1}{\partial^{j-m}\left(e^{(\Phi - \bar{\Phi}) /h}A'_{j,k}\bar{\partial}^{k-m}\left(e^{(\bar{\Phi}-\Phi)/h}g\right)\right)}
      -\sum_{j,k = 0}^{m-1}{\partial^{j-m}\left(e^{(\Phi - \bar{\Phi}) /h}A'_{j,k}\bar{\partial}^{k}a\right)}. 
    \tag{\ref{eqng}*}
\end{equation}
Using the fact that $g \in L^2(\Omega)$ and $A'_{j,k} \in W^{j+k+1,p}(\Omega)$ in the right hand side of \eqref{eqng2} coupled with the mapping properties of the Cauchy operators (Lemma \ref{MPCO}), one can show that $g \in H^1(\Omega).$ Repeating the above argument now with $g \in H^1(\Omega)$ in the right hand side of \eqref{eqng2}, one obtains that $g \in H^2(\Omega)$. We proceed with this recursive argument and increase the regularity of $g$ with each repetition. Note that we can not apply the argument infinite times as the second term in the right hand side of \eqref{eqng2} has limited regularity of $H^{m+1}(\Omega)$. Thus, we conclude that $g \in H^m(\Omega).$ 
}
\end{proof}
\vspace{0.4cm}

\noindent \textbf{Proof of Theorem \ref{CGO}.}  Choose $g$ as in Proposition \ref{g} and let
\[
   r_h = \bar{\partial}^{-m}\left(e^{(\bar{\Phi}-\Phi)/h}g\right)
\]
as observed in \eqref{fsoln}. Clearly, $r_h \in H^{2m}(\Omega)$ and $\norm{r_h}_{H^{m}(\Omega)} = O(h^{\frac{1}{2}+\varepsilon})$. Then we see that $u = e^{\Phi/h}(a + r_h) \in H^{2m}(\Omega)$ where $\bar\partial^m a =0$ solves $\mathcal{L}u=0.$ This proves Theorem \ref{CGO}. \qed \\

The adjoint operator $\mathcal{L}^{*}$ has a similar form as the operator $\mathcal{L}.$ Hence, by following similar arguments, we can show that the adjoint equation has the same type of CGO solutions as given in Theorem \ref{CGO}.

\begin{remark}
    One can define an integral equation for $f$ instead of $g$ by substituting \eqref{gsoln} in \eqref{fsoln} and  following the above procedure one can obtain that $\norm{r_h}_{H^m(\Omega)} = O(h^{\frac{1}{2} - \varepsilon}).$ 
\end{remark}

\begin{remark} From the techniques used in the proof of Theorem \ref{CGO} one can even construct CGO solutions to the equation
    \[
    \mathcal{P}u  = \mathcal{L}u + \sum_{j = 0}^{m-1}{A_{j,m}\partial^{j}\bar\partial^{m}u} = \partial^{m}\bar{\partial}^{m}u   + \sum_{j,k = 0}^{m - 1}{A_{j,k}\partial^{j}\bar{\partial}^{k}u} + \sum_{j = 0}^{m-1}{A_{j,m}\partial^{j}\bar\partial^{m}u} = 0
    \]   
of the form \eqref{soln}.
But the difficulty lies in the construction of CGO solutions to the adjoint $P^*$ as the transport equation defining the amplitude of the CGO solution to $P^*v=0$ will depend on the coefficients $A_{j,m}.$
\end{remark}

\section{Proof of uniqueness of coefficients}\label{III}
\noindent In this section, we prove Theorem \ref{MT}. We start by deriving an integral identity.
\noindent Let $u_{1},v \in H^{2m}(\Omega)$ be such that
\begin{equation}\label{Integral2}
\mathcal{L}u_{1} = 0, \quad \tilde{\mathcal{L}}^{*}v = 0.
\end{equation}
By assuming $\mathcal{C}(\mathcal{L}) = \mathcal{C}(\tilde{\mathcal{L}})$ there exists $u_{2}\in H^{2m}(\Omega)$ satisfying 
\begin{equation}\label{u2eq}
		\begin{rcases}
		\begin{aligned}
  \tilde{\mathcal{L}}u_{2} & = 0\\
		u_2|_{\bd} &= u_1|_{\bd},  \\
  (\partial_\nu u_2)|_{\bd} &= (\partial_\nu u_1)|_{\bd}, \\
  \vdots \quad \quad &  \quad  \quad \vdots\\
(\partial_\nu^{(2m-1)} u_{2})|_{\bd} &= (\partial_\nu^{(2m-1)} u_{1})|_{\bd}.
		\end{aligned}
		\end{rcases}
	\end{equation}
Note that
\[		
\tilde{\mathcal{L}}(u_{1} - u_{2}) =  \sum_{j,k = 0}^{m-1}(\tilde{A}_{j,k} -  A_{j,k})\partial^{j}\bar{\partial}^{k}u_{1}.
\]
Now we use integration by parts and \eqref{u2eq} to obtain the following integral identity
\begin{align*}
0 & = \int_{\Omega}{(u_{1} - u_{2})\overline{\tilde{\mathcal{L}}^{*}v}\, dx}\\
& = \int_{\Omega}{\tilde{\mathcal{L}}(u_{1} - u_{2})\bar{v}\, dx}\\
& = \int_{\Omega}{\left[\sum_{j,k = 0}^{m-1}(\tilde{A}_{j,k} -  A_{j,k})\partial^{j}\bar{\partial}^{k}u_{1}\right]\bar v\, dx}.
\end{align*}
By our assumption \eqref{zerocond}, we get the following integral identity
\begin{equation}
  \sum_{j,k = 0}^{m-1}\left((-1)^{j}\int_{\Omega}{({\tilde{A}'}_{j,k} -  {A'}_{j,k})\bar{\partial}^{k}u_{1}\partial^{j}\bar{v}\, dx}\right) = 0. 
\end{equation}
where $\mathcal{L}(u_{1}) = 0$ and $\tilde{\mathcal{L}}^{*}(v) = 0$. \\

\noindent By using Theorem \ref{CGO} we consider $u_1$ and $v$ of the form
\begin{align}\label{uv}
\begin{split}
u_{1} & = e^{\Phi/h}(a + r_{h}),  \text{ where } \ \bar\partial^m a = 0,\\
v & = e^{-\Phi/h}(b + s_{h}),  \text{ where } \ \bar\partial^m b =0,
\end{split}
\end{align}
with $r_h$ and $s_h$ satisfy $\|r_{h}\|_{H^{m}(\Omega)} =  O(h^{\frac{1}{2} + \epsilon})$ and $\|s_{h}\|_{H^{m}(\Omega)} =  O(h^{\frac{1}{2} + \epsilon})$ for some $\epsilon > 0$.\\

\noindent By using $u_1$ and $v$, the integral identity takes the form 
\begin{align*}
 0 & = \sum_{j,k = 0}^{m-1}(-1)^{j} \int_{\Omega}{\left[e^{(\Phi - \bar{\Phi})/h}(\tilde{A}'_{j,k}- A'_{j,k}) \bar{\partial}^{k}a \partial^{j}\bar{b}\right] }\\
  & \quad + \sum_{j,k = 0}^{m-1}(-1)^{j} \int_{\Omega}{\left[e^{(\Phi - \bar{\Phi})/h}(\tilde{A}'_{j,k}- A'_{j,k}) (\bar{\partial}^{k}a \partial^{j}\bar{s}_{h} +\bar{\partial}^{k}r_{h} \partial^{j}\bar{b} + \bar{\partial}^{k}r_{h} \partial^{j}\bar{s}_{h}) \right]}
\end{align*}

{\color{black}
    \noindent From the assumption \eqref{zerocond} of Theorem \ref{MT}, we have the following identity
    \begin{multline}\label{UOC4}
    \sum_{j,k = 0}^{m-1}  \left( (-1)^j \int_{\Omega}{{ e^{(\Phi - \bar{\Phi})/h}(\tilde{A}'_{j,k}- A'_{j,k}) (\bar{\partial}^{k}a) (\partial^{j}\bar{b})}} \right)\\ = \sum_{j,k = 0}^{m-1}{C_{j,k}(z_{0})h e^{(\Phi(z_0) - \bar{\Phi}(z_0))/h}(\tilde A'_{j,k}(z_0) - {A}'_{j,k}(z_0))\bar{\partial}^{k}a(z_{0}) \partial^{j}\bar{b}(z_{0})} + o(h)
\end{multline}
where $C_{j,k}(z_{0})\neq 0$ for all $0\leq j,k \leq m-1$.
    \noindent To obtain \eqref{UOC4}, we proceed as in \cite[Theorem 3.3]{GT}.
    Let $B_{j,k} = \tilde{A}'_{j,k}- A'_{j,k} \in W^{j+k+1,p}(\Omega) \subset C^{j+k}(\bar{\Omega})$. We decompose $B_{j,k}$ as $\chi B_{j,k} + (1-\chi)B_{j,k}$ where $\chi \in C_0^\infty(\Omega)$ with $\chi(z_0)=1.$ Using the method of stationary phase we obtain the following identity
\begin{equation}\label{SPEq1}
    \begin{aligned}
    \sum_{j,k = 0}^{m-1}  &\left( (-1)^j \int_{\Omega}{{ e^{(\Phi - \bar{\Phi})/h} \chi B_{j,k} (\bar{\partial}^{k}a) (\partial^{j}\bar{b})}} \right)\\ 
    &= \sum_{j,k = 0}^{m-1}{C_{j,k}(z_{0})h e^{(\Phi(z_0) - \bar{\Phi}(z_0))/h} \chi(z_0) B_{j,k}(z_0)\bar{\partial}^{k}a(z_{0}) \partial^{j}\bar{b}(z_{0})} + o(h)\\
    &= \sum_{j,k = 0}^{m-1}{C_{j,k}(z_{0})h e^{(\Phi(z_0) - \bar{\Phi}(z_0))/h}(\tilde A'_{j,k}(z_0) - {A}'_{j,k}(z_0))\bar{\partial}^{k}a(z_{0}) \partial^{j}\bar{b}(z_{0})} + o(h).
\end{aligned}
\end{equation} 
    Next we use integration by parts to show that the contribution from the terms $(1-\chi)B_{j,k}$  is $o(h).$ Let $\{\chi_l\}$ be the partition of unity on $\Omega$ with $U_l = \text{supp}\{\chi_l\}$ and $\sum_l \chi_l =1$.   
\begin{equation}\label{SPEq2}
    \begin{aligned}
    \sum_{j,k = 0}^{m-1}  &\left( (-1)^j \int_{\Omega}{{ e^{(\Phi - \bar{\Phi})/h} \chi_l (1-\chi) B_{j,k} (\bar{\partial}^{k}a) (\partial^{j}\bar{b})}} \right)\\ 
    &= h \sum_{j,k = 0}^{m-1} \left( (-1)^j \int_{U_l}{{ e^{(\Phi - \bar{\Phi})/h} \partial \left( \frac{\chi_l (1-\chi) B_{j,k} (\bar{\partial}^{k}a) (\partial^{j}\bar{b})}{\partial (\Phi - \bar{\Phi})} \right)}} \right).
\end{aligned}
\end{equation}   
Since $\partial \left( \frac{\chi_l (1-\chi) B_{j,k} (\bar{\partial}^{k}a) (\partial^{j}\bar{b})}{\partial (\Phi - \bar{\Phi})} \right) \in L^1(\Omega)$, using the assumption \eqref{zerocond} and Riemann–Lebesgue lemma we conclude that the right-hand side of \eqref{SPEq2} is $o(h)$. Thus, \eqref{SPEq1}-\eqref{SPEq2} give \eqref{UOC4} .


}

\noindent Next, we use the fact that $\|r_{h}\|_{H^{m}(\Omega)} =  O(h^{\frac{1}{2} + \epsilon})$,  $\|s_{h}\|_{H^{m}(\Omega)} =  O(h^{\frac{1}{2} + \epsilon}),$ for some $\epsilon > 0$ and obtain the following estimate
\begin{equation}\label{UOC1}
\sum_{j,k = 0}^{m-1}(-1)^{j} \int_{\Omega}{\left[e^{(\Phi - \bar{\Phi})/h}(\tilde{A}'_{j,k}- A'_{j,k})\bar{\partial}^{k}r_{h} \partial^{j}\bar{s_{h}} \right]} = O(h^{1 + 2\epsilon}).
\end{equation}

\noindent Let $\tilde{r}_h$ and $\tilde{s}_h$ be such that
 \[
    \begin{aligned}
        r_h &= \bar{\partial}^{-m}\left(e^{(\bar{\Phi}-\Phi)/h} \tilde{r}_h\right), \\ 
        s_h &= \bar{\partial}^{-m}\left(e^{(\bar{\Phi}-\Phi)/h} \tilde{s}_h\right)
    \end{aligned}
 \]
 and they satisfy the equation \eqref{eqng} in place of $g$. Using this we have the following estimate
\begin{align*}
   &  \sum_{j,k = 0}^{m-1}(-1)^{j} \int_{\Omega}{\left[e^{(\Phi - \bar{\Phi})/h}(\tilde{A}'_{j,k}- A'_{j,k})\bar{\partial}^{k}r_{h} \partial^{j}\bar{b} \right]}\\
     & \quad  = \sum_{j,k = 0}^{m-1}(-1)^{j} \int_{\Omega}{\left[e^{(\Phi - \bar{\Phi})/h}(\tilde{A}'_{j,k}- A'_{j,k})\bar{\partial}^{k-m}\left(e^{(\bar{\Phi}-\Phi)/h}\tilde{r}_h\right)\partial^{j}\bar{b} \right]}\\
     & \quad  = \sum_{j,k = 0}^{m-1}(-1)^{j} \int_{\Omega}{\left[\bar{\partial}^{k-m}\left(e^{(\Phi - \bar{\Phi})/h}(\tilde{A}'_{j,k}- A'_{j,k})\partial^{j}\bar{b}\right)e^{(\bar{\Phi}-\Phi)/h}\tilde{r}_h \right]}\\
     & \quad \leq C h^{\frac{1}{2}+ \epsilon}\sum_{j,k = 0}^{m-1} \norm{(\tilde A'_{j,k}-A'_{j,k})\partial^{j}\bar{b} }_{W^{1,p}}\norm{\tilde{r}_h }_{L^{2}},
\end{align*}
where we have used Fubini's theorem in the second equality while the last inequality is obtained by applying Lemma \ref{inverseestimate}. Now, we apply Proposition \ref{g} to obtain
\begin{equation}\label{UOC2}
   \sum_{j,k = 0}^{m-1}(-1)^{j} \int_{\Omega}{\left[e^{(\Phi - \bar{\Phi})/h}(\tilde{A}'_{j,k}- A'_{j,k})\bar{\partial}^{k}r_{h} \partial^{j}\bar{b} \right]} = O(h^{1 + 2\epsilon}) .
\end{equation}
Similarly, we obtain
\begin{equation}\label{UOC3}
   \sum_{j,k = 0}^{m-1}(-1)^{j} \int_{\Omega}{\left[e^{(\Phi - \bar{\Phi})/h}(\tilde{A}'_{j,k}- A'_{j,k})\bar{\partial}^{k}a \partial^{j}\bar{s}_{h} \right]} = O(h^{1 + 2\epsilon}) .
\end{equation}

\noindent \textbf{Proof of Theorem \ref{MT}.} Using the estimates \eqref{UOC4} - \eqref{UOC3} and matching the asymptotics as $h \to 0$, we obtain
\begin{equation}\label{UQ}
 0 = \sum_{j,k = 0}^{m-1} C_{j,k}(z_{0}) (\tilde{A}'_{j,k}(z_0) - {A}'_{j,k}(z_0)) \bar{\partial}^{k}a(z_{0}) \partial^{j}\bar{b}(z_{0}).
\end{equation}
We now show that $A'_{0,0}=\tilde{A}'_{0,0}$. To this end, let us choose $a = b = 1$. With this choice we obtain
\[
    \tilde A'_{0,0}(z_0) = {A}'_{0,0}(z_0).
\]
Since for any $z_{0}\in\Omega$ we can choose $\Phi$ with a unique critical point at $z_{0}$, we have
\[
    \tilde A'_{0,0} = {A}'_{0,0} \quad \text{in} \ \Omega.
\]
\noindent Next to show that $A'_{0,1} = \tilde{A}'_{0,1},$ we rewrite \eqref{UQ} by setting the term $\tilde A'_{0,0} - {A}'_{0,0} = 0.$ Then, we choose $a=\bar{z}, b=1$ to obtain
\[
    A'_{0,1} = \tilde{A}'_{0,1},\quad \text{in} \ \Omega.
\]
Similarly, one can show $A'_{j,k} = \tilde{A}'_{j,k}$ in an increasing order for $j + k$ by choosing
\[
    a= \frac{\bar{z}^k}{k!} \ \text{ and } b = \frac{\bar{z}^j}{j!}\]
and applying the above procedure to obtain
\[
    \tilde{A}'_{j,k} = {A}'_{j,k} \ \text{in} \ \Omega \quad \text{for all} \ \ 0\leq j,k \leq m-1.
\]
From \eqref{eqn7}, we readily obtain that
\[
    \tilde{A}_{j,k} = {A}_{j,k} \ \text{in} \ \Omega \quad \text{for all} \ \ 0\leq j,k \leq m-1.
\]
This proves Theorem \ref{MT}.
\section*{Acknowledgements}
\noindent VPK would like to thank the Isaac Newton Institute for Mathematical Sciences, Cambridge, UK, for support and hospitality during \emph{Rich and Nonlinear Tomography - a multidisciplinary approach} in 2023 where part of this work was done (supported by EPSRC Grant Number EP/R014604/1). The authors thank Manas Kar for suggesting this problem and Masaru Ikehata for drawing our attention to the references \cite{Ik91,Ik93,Ik94}.

\bibliographystyle{alpha}
\bibliography{RA}

\end{document}